\documentclass[12pt]{amsart}
\usepackage{amsmath,amssymb,amscd,amsthm,amscd,
mathrsfs,inputenc,enumerate,amsfonts,bbold,pdfsync}

\usepackage[english]{babel}
\usepackage{color}%

\newcommand{\spec}{{\rm Spec}}

\newcommand{\zar}{{\rm Zar}}

\newcommand{\ms}{\mathscr}
\newcommand{\mc}{\mathcal}
\newcommand{\ad}{{\rm Ad}}
\newcommand{\f}{\mathfrak}

\newcommand{\clop}{{\rm Clop}}

\newcommand{\z}{{\ldots}}
\newcommand{\w}{{\setminus}}

\newtheoremstyle{break1}
  {9pt}
  {9pt}
  {}
  {}
  {\textbf}
  {.}
  {.7em}
  {}
\newtheorem{thm}{\textbf{Theorem}}[section]
\newtheorem{defn}[thm]{\textbf{Definition}}
\newtheorem{cor}[thm]{\textbf{Corollary}}
\newtheorem{lem}[thm]{\textbf{Lemma}}
\newtheorem{prop}[thm]{\textbf{Proposition}}

\theoremstyle{definition}
\newtheorem{ex}[thm]{ \textbf{Example}}
\newtheorem{oss}[thm]{Remark}
\theoremstyle{remark}

\begin{document}

\title[]
{Spectral spaces and ultrafilters}
\author{Carmelo Antonio Finocchiaro}
\maketitle
\begin{flushright}
\scriptsize \it In memory of my father
\end{flushright}
\begin{abstract}
Let $X$ be the prime spectrum of a ring. In \cite{folo} the authors define a topology on $X$ by using ultrafilters and they show that this topology is precisely the constructible topology. In this paper we generalize the construction given in \cite{folo} and, starting from a set $X$ and a collection of subsets $\mathcal{F}$ of $X$, we define by using ultrafilters a topology on $X$ in which $\mathcal F$ is a collection of clopen sets. We use this construction for giving a new characterization of spectral spaces and several examples of spectral spaces. 
\end{abstract}

\section*{Introduction}
Let $A$ be a ring (with this term we  will always mean a commutative ring with multiplicative identity) and let $\spec(A)$ be the prime spectrum of $A$, i.e., the set of all prime ideals of $A$. As it is well known, $\spec(A)$ has a natural structure of topological space, whose closed sets are the sets of the form $V(\f a):=\{\f p\in \spec(A):\f p\supseteq \f a\}$, where $\f a$ runs in the collection of all ideals of $A$. The topology on $\spec(A)$ obtained in this way is the so called \emph{Zariski topology}. A topological space is called \emph{spectral} if it is homeomorphic to $\spec(A)$ (endowed with the Zariski topology), for some ring $A$. In \cite{ho} the author gives a characterization  of spectral spaces. More precisely, he shows that a topological space is spectral if and only if it is compact (i.e., every open cover has a finite subcover), it has a basis of open and compact subspaces that is closed under finite intersections, and every irreducible closed subset has a unique generic point (i.e., it is the closure of a unique point).  

Thus the Zariski topology on the prime spectrum of a ring $A$ is always  T$_0$, but it is almost never Haudorff. More precisely, it satisfies the Hausdorff axiom if and only if the Krull dimension of $A$ is 0. Thus, the following question arises naturally. How can the Zariski topology be refined in order to make $\spec(A)$ an Hausdorff space without losing compactness? The following classical construction (see \cite[Chapter 3, Exercises 27--30]{AM} and \cite{ho}) answers completely this question: consider the set $\spec(A)$ and let $\mathcal P:=\{D_f:=\spec(A)\setminus V(fA):f\in A\}$ denote the collection of the so called \emph{principal open subsets} of $\spec(A)$ (this is clearly a basis of the Zariski topology). Consider on $\spec(A)$ the coarsest topology for which $\mathcal P$ is a collection of clopen sets. This new topology is known as \emph{the constructible} (or  \emph{patch}) topology on $\spec(A)$, and it makes $\spec(A)$ a compact and Hausdorff space. The reason of the name \emph{constructible} is mainly historical: indeeed, when $\spec(A)$, with the Zariski topology, is a Noetherian space, then the clopen subsets of the constructible topology are precisely the \emph{constructible subsets}, in the sense of Chevalley (i.e., are finite unions of locally closed subsets of $\spec(A)$). 

In a very recent paper by Fontana and Loper (see \cite{folo}) it is shown that the constructible topology is identical to ``another'' topology, defined by using ultrafilters. More precisely, let $Y$ be a subset of $\spec(A)$ and let $\ms U$ be an ultrafilter on $Y$. According to \cite[Lemma 2.4]{calota}, the subset $\f p_{Y,\ms U}:=\{x\in A:V(x)\cap Y\in \ms U\}$ of $A$ is a prime ideal of $A$, called \emph{the ultrafilter limit point of $Y$, with respect to $\ms U$}. Ultrafilter limit points of a subset $Y$ of $\spec(A)$ are not always elements of $Y$: for example, if $A$ is the ring of integers, $Y:={\rm Max}(A)$ and  $\ms U$ is a nontrivial ultrafilter on $Y$ (i.e., an ultrafilter whose elements are infinite sets), then it is very easy to check that $\f p_{Y,\ms U}=(0)$. Thus the following definition is natural: define a subset $Y$ of $\spec(A)$ to be \emph{closed under ultrafilters} if it contains all its ultrafilter limit points, with respect each ultrafilter on $Y$. Then, in \cite{folo} it is proved that the ultrafilter closed sets of $\spec(A)$ are the closed sets for a topology on $\spec(A)$, called \emph{the ultrafilter topology}, and it is shown that the  constructible topology and the ultrafilter topology are the same topology. Recently in \cite{fifolo1} and \cite{fifolo2} the authors gave  several applications of the constructible topology on the spectral space $\zar(K|A)$ of all the valuation domains of a field $K$ containing a fixed subring $A$ of $K$. Other important contributions on the algebraic applications of the topological properties of the Riemann--Zariski space $\zar(K|A)$ were given for instance in \cite{ol} and \cite{ol1}.

Inspired by the idea given in \cite{folo}, the first goal of this paper is to define the ultrafilter topology in a more general setting. Let $X$ be a set and $\mathcal F$ be a nonempty collection of subsets of $X$. To any subset $Y$ of $X$ and any ultrafilter $\ms U$ on $Y$, we can associate a  set, namely $Y_{(X,\mathcal F)}(\ms U)$, whose points are the points $x$ of $X$ such that, for any member $F\in\mathcal F$, we have $x\in F$ if and only if $F\cap Y\in \ms U$. As the reader will see in Example \ref{gioco}(\ref{giocospec}), in this general setting $X$ plays the role of the space of prime ideals $\spec(A)$ of a ring $A$, $\mathcal F$ the role of the collection $\mathcal P$ of the principal open subsets of $\spec(A)$, and $Y_{(X,\mathcal F)}(\ms U)$ the role of the ultrafilter limit point $\f p_{Y,\ms U}$ of $Y$, with respect to $\ms U$. Then we  define the notion of $\mathcal F-$stable under ultrafilters subset of $X$ (see Definition \ref{stable}) and we show that the sets that are $\mathcal F-$stable under ultrafilters are the closed sets for a topology. Thus, given a set $X$ and a collection of subsets $\mathcal F$ of $X$, we can define a new topology, and we call it the $\mathcal F-$ultrafilter topology. After several examples, we  discuss the main properties of this kind of topology. For example, we  show that in this topology $\mathcal F$ is always a collection of clopen sets. Also we describe a closure of a set and characterize when the $\mathcal F-$ultrafilter topology is compact. From this we deduce the equality of the constructible topology and the ultrafilter topology in the particular case in which $X$ is the prime spectrum of a ring (\cite[Theorem 8]{folo}). In Section 3 we give an application of the general theory developed before. More precisely, we study the case in which $(X, \mathcal T)$ is a topological space and $\mathcal F$ is a basis of this topology. In Proposition \ref{B-ultra} we describe how the original topology $\mathcal T$ and the $\mathcal F-$ultrafilter topology are related. In Corollary \ref{spectralcriterion} we deduce a new criterion, based on the use of ultrafilters, to decide when a topological space is spectral. In Propositions \ref{overrings} and \ref{icoverrings} we apply it to give new examples of spectral spaces. Finally, in Theorem \ref{scheme-constructible} we show that the constructible topology defined by Grothendiek (see \cite[pag. 337, (7.2.11)]{EGA}) on the underlying topological space of a scheme can be seen as a particular case of the construction introduced in Section 2. 
\section{Notation and preliminaries}
We begin by giving some notation and preliminary results.  If $X$ is a set, the collection of all the subsets (resp., finite subsets) of $X$ will be denoted by $\mc B(X)$ (resp., $\mc B_{\tt fin}(X)$). If $\mc F\subseteq \mc B(X)\setminus\{\emptyset\}$, we shall denote by $\bigcap \mc F$ (resp., $\bigcup \mc F$)  the intersection (resp., the union) of all  members of $\mc F$. If $X$ is a topological space, then we will denote by $\clop(X)$ the collection of all the clopen subsets of $X$. If $Y$ is a subset of $X$, we will denote by $\ad(Y)$ the closure of $Y$.

Recall that given a set $X$ and a nonempty collection $\ms F$ of subsets of $X$, we say
 that $\ms F$ is \emph{a filter on} $X$ if the following properties hold:
\begin{enumerate}[(i)]
\item $\emptyset\notin \ms F$. 
\item If $Y,Z\in \ms F$, then $Y\cap Z\in \ms  F$. 
\item If $Z\in \ms F$ and $Z\subseteq Y\subseteq X$, then $Y\in \ms F$.
\end{enumerate}
A filter  $\ms F$ is \emph{an ultrafilter on} $X$ if $\ms F$ is a maximal element (under inclusion) in the set of all filters on $X$. We shall denote by $\beta X$ the set of all the ultrafilters on $X$. \\\\
Now we state without proof some easy and well known properties of ultrafilters on sets (see \cite{je}). 
\begin{oss}\label{beginning}
Let $X$ be a set.
\begin{enumerate}[\rm (i)]
\item If $x\in X$, then the collection of sets 
$${\beta_X}^x:=\{Y\subseteq X:x\in Y\}$$ 
is an ultrafilter on $X$, and it is called {\it the trivial ultrafilter generated by }$x$.
\item\label{ultracaratt} If  $\ms F$ is a filter on  $X$, the following conditions are equivalent:
\begin{enumerate}[\rm (a)]
\item $\ms F$ is an ultrafilter on $X$.
\item For each subset $Y$ of $X$, then $Y\notin \ms F$ implies $X\setminus Y\in \ms F$. 
\item If $Y$ and $Y_0$ are subsets of $X$ such that $Y\cup Y_0\in \ms F$, then $Y\in \ms F$ or $Y_0\in \ms F$. 
\end{enumerate}
\item\label{zornultra} If $\ms F$ is a collection of subsets of $X$ with finite intersection property, then there exists an ultrafilter $\ms U$ on $X$ containing $\ms F$ (by a straightforward application of Zorn's Lemma).
\item\label{giu} If $\ms U$ is an ultrafilter on $X$ and $Y\in \ms U$, then 
$$
\ms U_{ Y}:=\{Y\cap U:U\in\ms U\}
$$
is an ultrafilter on $Y$ contained in $\ms U$.
\item\label{su}  If $f:X\longrightarrow Y$ is a function and $\ms U$ is an ultrafilter on $X$, then $\ms U^f:=\{T\subseteq Y:f^{-1}(T)\in\ms U\}$ is an ultrafilter on $Y$.\\
In particular, If $Z\subseteq Y\subseteq X$ and $\ms U$ is an ultrafilter on $Z$, then 
$$
\ms U^{ Y}:=\{T\subseteq Y:T\cap Z\in \ms U\}
$$
is an ultrafilter on $Y$ containing $\ms U$. In fact, if $\iota:X\longrightarrow Y$ is the inclusion, then $\ms U^Y=\ms U^\iota$.
\end{enumerate} 
\end{oss}
\section{The construction}
 Let $X$ be a set and $\mc F$ be a nonempty collection of subsets of $X$. For each $Y \subseteq X$ and each ultrafilter $\ms U$ on  $Y$, we define
$$ 
Y_{(X,\mathcal F)}(\ms U):=Y_{\mc F}(\ms U):=\{x\in X:[\forall F\in \mc F, x\in F\Longleftrightarrow F\cap Y\in \ms U]\}.
$$
Since $\mc F$ will be almost always a fixed collection of subsets of a fixed set $X$, we will denote the set $Y_{(X,\mc F)}(\ms U)$ simply by $Y(\ms U)$, when no confusion can arise.
\begin{ex}\label{gioco}
 Let $X$ be a set, $\mc F$ be a nonempty collection of subsets of $X$ and $Y$ be a subset of $X$. 
\begin{enumerate}[(1)]
\item If  $y\in Y$ and $\beta^y_Y$ is, as usual, the trivial ultrafilter on $Y$ generated by $y$, then $y\in Y_{\mc F}(\beta^y_Y)$. 
\item\label{giocospec} Let $A$ be a ring, $Y$ be a subset of $\spec(A)$ and $\ms U$ be an ultrafilter on $Y$. Set, as before, 
$$
\f p_{\ms U}:=\{x\in A:V(x)\cap Y\in \ms U\}.
$$
Then, if $\mc P:=\{D_a:a\in A\}$ is the collection of all the principal open subsets of $\spec(A)$, the equality $Y_{\mc P}(\ms U)=\{\f p_{\ms U}\}$ holds. As a matter of fact, fix $\f p\in \spec(A)$. Then, by definition, $\f p\in Y_{\mc P}(\ms U)$ if and only if the following statement is true:
$$
\mbox{ for any }a\in A, \f p\in D_a \Longleftrightarrow D_a\cap Y\in \ms U. 
$$
Obviously, the previous statement can be written in the following (equivalent) way:
$$
\mbox{ for any }a\in A, a\in \f p \Longleftrightarrow V(a)\cap Y\in \ms U. 
$$
Now it follows immediately that $\f p\in Y_{\mc P}(\ms U)$ if and only if $\f p=\f p_{\ms U}$. 
\item\label{giocozar} Let $K$ be a field and $A$ be a subring of $K$. Denote by $\zar(K|A)$ the set of all the valuation domains of $K$ containing $A$ as a subring. Let $\mathcal Q:=\{B_F:=\zar(K|A[F]):F\in \mc B_{\tt fin}(K)\}$ be the natural basis of open sets of the Zariski topology of $\zar(K|A)$. If $Y$ is a subset of $\zar(K|A)$ and $\ms U$ is an ultrafilter on $Y$, set
$$
A_{\ms U}:=\{x\in K:B_{\{x\}}\cap Y\in \ms U \}
$$
Then, as in (2), we have $Y_{\mc Q}(\ms U)=\{A_{\ms U}\}$. 
\end{enumerate}
\end{ex}
\begin{oss}
Let $X$ be a set and $\mc F$ be a nonempty collection of subsets of $X$ that is closed under complements. Then, for any subset $Y$ of $X$ and any ultrafilter $\ms U$ on $Y$, we have
$$
Y_{\mc F}(\ms U)=\bigcap \{F\in \mc F: F\cap Y\in\ms U\}.
$$
The inclusion $\subseteq$ is clear, by  definition. Conversely, take an element $x\in \bigcap \{F\in\mc F:F\cap Y\in \ms U\}$, and fix a set $F\in \mc F$. If $Y\cap F\in\ms U$, then $x\in F$, by assumption. Conversely, suppose that $x\in F$. If $Y\cap F\notin\ms U$, then $Y\w F=(X\w F)\cap Y\in \ms U$, since $\ms U$ is an ultrafilter on $Y$. Also, by assumption, $X\w F\in \mc F$ and thus $x\in X\w F$, being $x$ an element of the intersection, and this is a contradiction.  
\end{oss}
\begin{defn}\label{stable}
Let $X$ be a set and $\mc F$ be a nonempty collection of subsets of $X$. Then, we say that a subset $Y$ of $X$ is {\rm $\mc F-$stable under ultrafilters} if $Y(\ms U)\subseteq Y$, for each ultrafilter $\ms U$ on $Y$. 
\end{defn}
\begin{ex} Let $A$ be a ring (resp. $K$ be a field and $S$ be a subring of $K$). Keeping in mind the notation and the statements  given in Example \ref{gioco} it follows immediately that a subset $Y$ of $\spec(A)$ (resp. $\zar(K|A)$) is $\mc P-$stable under ultrafilters (resp. $\mc Q-$stable under ultrafilters) if and only if it is closed in the ultrafilter topology of $\spec(A)$ (resp. $\zar(K|A)$). 
\end{ex}
The following easy and technical lemma will allow us to show that the ultrafilter topology is a very particular case of a more general construction. 
\begin{lem}\label{technical}
Let $X$ be a set, $\mc F$ be a given nonempty collection of subsets of $X$ and $ Y\subseteq Z\subseteq X$. Let  $\ms U$ be an ultrafilter on $Y$, $T\in \ms U$ and, as in Remark  \ref{beginning}(iv,v), set
$$
\ms U_T:=\{U\cap T:U\in\ms U\}\qquad \ms U^Z:=\{Z'\subseteq Z:Z'\cap Y\in \ms U\}
$$
Then we have 
$$ 
Y(\ms U)=T(\ms U_T)=Z(\ms U^Z).
$$
\end{lem}
\begin{proof} We shall prove only the inclusion $Y(\ms U)\subseteq T(\ms U_T)$. The others are shown with the same straightforward arguments. Let $x\in Y(\ms U)$ and $F\in \mc F$. We need to show that $x\in F$ if and only if $F\cap T\in \ms U_T$. Assume $x\in \mc F$. Then, $F\cap Y\in \ms U$ and $F\cap T=(F\cap Y)\cap T\in \ms U_T$, by definition. Assume $F\cap T\in \ms U_T$. Since $\ms U_T\subseteq \ms U$ and $F\cap T\subseteq F\cap Y$, then $F\cap Y\in \ms U$ and thus $x\in F$. 
\end{proof}
\begin{prop}\label{F-ultratop}
Let $X$ be a set and $\mc F$ be a nonempty collection of subsets of $X$. Then, the family of all the subsets of $X$ that are $\mc F-$stable under ultrafilters is the collection of the closed sets for a topology on $X$. We will call it {\rm the $\mc F-$ultrafilter topology on $X$}, and denote by $X^{^{\mc F-{\rm ultra}}}$ the set $X$ endowed with the $\mc F-$ultrafilter topology.
\end{prop}
\begin{proof}
Let $C,C_0$ be  $\mc F-$stable under ultrafilters subsets of $X$, and $\ms U$ be an ultrafilter on $Y:=C\cup C_0$. By Remark \ref{beginning}(ii), we can assume that $C\in \ms U$. Then, by hypothesis and Lemma \ref{technical}, we have $Y(\ms U)=C(\ms U_C)\subseteq C\subseteq Y$, and thus $Y$ is $\mc F-$stable under ultrafilters. 

Now, let $\mc G$ be a collection of  $\mc F-$stable  under ultrafilters subsets of $X$ and let $\ms U$  be an ultrafilter on $Z:=\bigcap \mc G$. For each $C\in \mc G$, we have $C(\ms U^C)=Z(\ms U)$ (by Lemma \ref{technical}), and thus $Z(\ms U)\subseteq Z$. This completes the proof.
\end{proof}
\begin{oss}\label{esempi1}
Let $X$ be a set.
\begin{enumerate}[\rm (1)]
\item The $\mc B(X)-$ultrafilter topology on $X$ is the discrete topology on $X$.
\item The $\{X\}-$ultrafilter topology on $X$ is the chaotic topology.
\item \label{exspec} Let $A$ be a ring and $\mathcal P$ be the collection of all the principal open subsets of $X:=\spec(A)$. Then, the $\mathcal P-$ultrafilter topology of $X$ is equal to the ultrafilter topology studied in \cite{folo}.
\item\label{exzar} Let $K$ be a field, $A$ be a subring of $K$ and $\mathcal Q$ be the natural basis of open sets for the Zariski topology on $\zar(K|A)$. Then, the $\mc Q-$ultrafilter topology is equal to the ultrafilter topology on $\zar(K|A)$. 
\item\label{finer} If $\mc F\subseteq \mc G$ are collections of subsets of $X$, then the $\mc G-$ultrafilter topology is finer than or equal to the $\mc F-$ultrafilter topology. In fact, for each subset $Y$ of $X$ and each ultrafilter $\ms U$ on $Y$, we have $Y_{\mc G}(\ms U)\subseteq Y_{\mc F}(\ms U)$. 
\end{enumerate}
\end{oss}
\begin{prop}\label{prebool}
Let $X$ be a set and $\mc F$ be a nonempty collection of subsets of $X$. Set 
$$\mc F_\sharp:=\left\{\bigcap\mc  G:\mc G\in\mc B_{\tt fin}(\mc F)\right\} \quad 
\mc F^\sharp:=\left\{\bigcup \mc G:\mc G\in\mc B_{\tt fin}(\mc F)\right\}
$$
$$
 \mc F^-:=\{X\w F:F\in \mc F\} .
$$
Then, the $\mc F-$ultrafilter topology, the $\mc F_{\sharp}-$ultrafilter topology and the $\mc F^\sharp-$ultrafilter topology are the same.
\end{prop}
\begin{proof}
By  Remark \ref{esempi1}(\ref{finer}) and the obvious inclusion $\mc F\subseteq \mc F^\sharp$, it is enough to show that the $\mc F^\sharp-$ultrafilter topology is finer or equal than the $\mc F-$ultrafilter topology. Let $Y$ be an $\mc F^\sharp-$stable under ultrafilter subset of $X$, $\ms U$ be an ultrafilter on $Y$, $x\in Y_{\mc F}(\ms U)$, $\mc G:=\{F_1,\z,F_n\}\in\mc B_{\tt fin}(\mc F)$ and $G:=\bigcap \mc G$. We want to show that $x\in G$ if and only if $G\cap Y\in \ms U$. If $x\in G$, then $F_i\cap Y\in \ms U$, for $i=1,\z, n$, and thus $G\cap Y\in \ms U$. Since $G\cap Y\subseteq F_i\cap Y$, for each $i=1,\z, n$, if $G\cap Y\in \ms U$, then it follows immediately that $F_i\cap Y\in\ms U$, for $i=1,\z,n$, and thus $x\in G$, by definition. This proves that $Y_{\mc F}(\ms U)\subseteq Y_{\mc F^\sharp}(\ms U)$. Thus it is clear that the $\mc F-$ultrafilter topology and the $\mc F_\sharp-$ultrafilter topology are the same. By a similar argument it can be shown that $Y_{\mc F}(\ms U)=Y_{\mc F^\sharp}(\ms U)=Y_{\mc F^-}(\ms U) $. Thus the proof is complete.
\end{proof}
\begin{cor}\label{bool}
Let $X$ be a set, $\mathcal F$ be a nonempty collection of subsets of $X$ and ${\rm Bool}(\mc F)$ be the boolean subalgebra of $\mc B(X)$ generated by $\mc F$. Then the $\mc F-$ultrafilter topology and the ${\rm Bool}(\mc F)-$ultrafilter topology are the same. 
\end{cor}
\begin{proof}
Let $Y$ be a nonempty subset of $X$ and $\ms U$ be an ultrafilter on $Y$. Keeping in mind the proof of Proposition \ref{prebool}, it follows that $Y_{\mc F}(\ms U)=Y_{\mc F\cup \mc F^-}(\ms U)$. Thus the $\mc F-$ultrafilter topology and the $(\mc F\cup \mc F^-)-$ultrafilter topology are the same. Since, obviously, ${\rm Bool}(\mc F)=\left(\left(\mc F\cup  \mc F^-\right)_\sharp\right)^\sharp$, the statement follows by using Proposition \ref{prebool}.
\end{proof}
\begin{ex}\label{opencompact}
Preserve the notation given in Remark \ref{esempi1}(\ref{exspec}), and let $\mathcal K$ be the collection of all the open and compact subspace of $X:=\spec(A)$. Since every element of $\mathcal K$ is a finite union of members of $\mathcal P$, it follows by Corollary \ref{bool} that the $\mathcal P-$ultrafilter topology and the $\mathcal K-$ultrafilter topology on $X$ are the same. 
\end{ex}
\begin{prop}\label{Fclopen}
Let $X$ be a set and $\mc F$ be a nonempty collection of subsets of $X$. Then, the following statements hold.
\begin{enumerate}[\rm (1)]
\item\label{Fclopen1} ${\rm Bool}(\mc F)\subseteq \clop(X^{^{\mc F-{\rm ultra}}})$.
\item\label{FclopenHausdorff} If, for each pair of distinct points $x,y\in X$ there exists a set $F\in \mc F$ such that $x\in F$ and $y\notin F$, then $X^{^{\mc F-{\rm ultra}}}$ is an Hausdorff and totally disconnected  space, and $Y(\ms U)$ has at most an element, for each $Y\subseteq X$ and $\ms U\in \beta Y$. 
\end{enumerate}
\end{prop}
\begin{proof}
Since $\clop(X^{^{\mc F-{\rm ultra}}})$ is a boolean algebra, it is enough to show that $\mc F\subseteq \clop(X^{^{\mc F-{\rm ultra}}})$.
Pick a set $E\in \mc F$. If $\ms U$ is an ultrafilter on $E$ and $x\in E(\ms U)$, then the statement $x\in F \Longleftrightarrow F\cap E\in \ms U$ holds for each $F\in \ms F$, and in particular for $F:=E$. Then, $x\in E$. Thus $E$ is closed in the $\mc F-$ultrafilter topology.

Now let $\ms V$ be an ultrafilter on $Z:=X\w E$ and $x\in Z(\ms V)$. The statement $x\in E\Longleftrightarrow E\cap Z\in \ms V$ holds and thus $x\in Z$. Then, $E$ is clopen. Thus (1) is  proved.

(2) The fact that $X^{^{\mc F-{\rm ultra}}}$ is an Hausdorff and totally disconnected space  follows immediately by (1) and the extra assumption on $\mathcal F$.  For the second part of (2), assume, by contradiction, that there exist distinct elements $x,y\in Y(\ms U)$, and pick, by hypothesis, a set $F\in\mc F$, such that $x\in F$ and $y\notin F$. Thus, $\emptyset=(F\cap Y)\cap (Y\w F)\in \ms U$.
\end{proof}
\begin{prop}
Let $X$ be a set, $\mc F$ be a nonempty collection of subsets of $X$, $\emptyset\neq Y\subseteq X$ and $\ms U$ an ultrafilter on $Y$. Then, for each topology on $X$ for which $\mc F$ is a collection of clopen sets,  $Y(\ms U)$ is closed. In particular, $Y(\ms U)$ is closed in the $\mc F-$ultrafilter topology.
\end{prop}
\begin{proof}
Let $x_0\in \ad(Y(\ms U))$ and $E\in\mc F$. If $x_0\in E$, then $E$ is an open neighborhood of $x_0$, by assumption, and thus there exists an element $y_1\in Y(\ms U)\cap E$. By definition, it follows that, for each $F\in\mc F$, $y_1\in F \Longleftrightarrow F\cap Y\in\ms U$, and thus $E\cap Y\in \ms U$, in particular. Conversely, assume $x_0\notin E$. Then, $X\w E$ is an open neighborhood of $x_0$, and thus there exists $y_2\in Y(\ms U)\w E$. Hence, we have $E\cap Y\notin \ms U$. This proves that $Y(\ms U)$ is closed.

The last part of the statement follows immediately by Proposition \ref{Fclopen}(\ref{Fclopen1}).
\end{proof}
\begin{prop}\label{closure}
Let $X$ be a set and $\mc F$ be a nonempty collection of subsets of $X$. Then, for each subspace $Y$ of $X^{^{\mc F-{\rm ultra}}}$, we have
$$
\ad(Y)=\bigcup\{Y(\ms U): \ms U\in \beta Y\}
$$
\end{prop}
\begin{proof}
Let $Y\subseteq X$, $\ms U\in \beta Y$,  $x\in Y(\ms U)$, and $\Omega$ be an open neighborhood of $x$. If $Y\cap \Omega=\emptyset$, then $Y\subseteq Z:=X\w\Omega$ and, using Lemma \ref{technical}, we have $Y(\ms U)=Z(\ms U^Z)\subseteq Z$, since $Z$ is $\mc F-$stable under ultrafilters. Thus we get a contradiction, since $x\in \Omega$. The inclusion $\supseteq$ follows. Conversely, pick an element $x\in \ad(Y)$, and set
$$
\mc G:=\{\Omega\cap Y:\Omega \mbox{ open neighborhood of } x\}.
$$
It is clear that $\mc G$ is a collection of subsets of $Y$ with the finite intersection property (since $x\in \ad(Y))$, and thus (by Remark \ref{beginning}(iii)) there exists an ultrafilter $\ms U^*$ on $Y$ containing $\mc G$. The conclusion will follow if we show that $x\in Y(\ms U^*)$. Fix $F\in\mc F$. If $x\in F$, then $F$ is an open neighborhood of $x$, by Proposition \ref{Fclopen}(\ref{Fclopen1}), and thus $F\cap Y\in \mc G\subseteq \ms U^*$. Conversely, assume $F\cap Y\in \ms U^*$. If $x\notin F$, then $X\w F$ is an open neighborhood of $x$, again by Proposition \ref{Fclopen}, and thus $(X\w F)\cap Y\in \mc G\subseteq \ms U^*$. It follows that $\emptyset\in \ms U^*$, a contradiction. 
\end{proof}
\begin{thm}\label{compactness}
Let $X$ be a set and $\mc F$ be a nonempty collection of subsets of $X$. Then, the following conditions are equivalent.
\begin{enumerate}[\rm(i)]
\item $X^{^{\mc F-{\rm ultra}}}$ is a compact topological space.
\item $X(\ms U)\neq \emptyset$, for each ultrafilter $\ms U$ on $X$.
\item If $\mc H$ is a subcollection of $\mc G:=\mc F\cup\mc F^-$ with the finite intersection property, then $\bigcap \mc H\neq \emptyset$.
\end{enumerate}
\end{thm}
\begin{proof}
(i)$\Longrightarrow$(iii). It is enough to use Proposition \ref{Fclopen}(\ref{Fclopen1}) and compactness of $X^{\mc F-{\rm ultra}}$. 

(iii)$\Longrightarrow$(ii). Let $\ms U$ be an ultrafilter on $X$. Assume, by contradiction, that $X(\ms U)=\emptyset$. This means that, for each $x\in X$ there exists a set $F_x\in\mc F$ such that exactly one of the following statements is true
\begin{enumerate}[(a)]
\item $x\in F_x$ and $F_x\notin \ms U$.
\item $x\notin F_x$ and $F_x\in \ms U$.
\end{enumerate}
Now, for each $x\in X$, set $C_x:=X\w F_x$, if $x\in F_x$, and $C_x:=F_x$, if $x\notin F_x$. Then, it is clear that $\mc H:=\{C_x:x\in X\}$ is a subcollection of $\mc G$ and that it has the finite intersection property, since $\mc H\subseteq \ms U$. Thus, by assumption, there exists $x_0\in \bigcap \mc H$. This is a contradiction, since $x\in X\w C_x$, for each $x\in X$.

(ii)$\Longrightarrow$(i). Let $\mc C$ be a collection of closed subsets of $X^{\mc F-{\rm ultra}}$ with the finite intersection property. By Remark \ref{beginning}(iii), there exists an ultrafilter $\ms U^*$ on $X$ such that $\mc C\subseteq \ms U^*$. By assumption, we can pick a point $x^*\in X(\ms U^*)$. Now, let $C\in \mc C$. Since $C\in \ms U^*$, we have $x^*\in X(\ms U^*)=C(\ms {U^*}_C)\subseteq C$, keeping in mind Lemma \ref{technical}. Thus $x^*\in \bigcap\mc C$. This completes the proof. 
\end{proof}
\begin{ex}\label{ultra-comp-spec}
Let $A$ be a ring. By Example \ref{gioco}(\ref{giocospec}), Remark \ref{esempi1}(\ref{exspec}) and Theorem \ref{compactness} we get immediately the well known fact that the ultrafilter topology on $\spec(A)$ is compact.
\end{ex}
\begin{prop}\label{coarsest}
Let $X$ be a set and $\mc F$ a nonempty collection of subsets of $X$ such that, for each couple of distinct points $x,y\in X$, there exists a set $F\in \mc F$ such that $x\in F$ and $y\notin F$. If $X^{^{\mc F-{\rm ultra}}}$ is a compact topological space, then the $\mc F-$ultrafilter topology is the coarsest topology for which $\mc F$ is a collection of clopen sets.
\end{prop}
\begin{proof} 
Denote by $X_\star$ the set $X$ with the coarsest topology for which $\mc F$ is a collection of clopen sets. Then, the identity map $\mathbf 1:X^{^{\mc F-{\rm ultra}}}\longrightarrow X_\star$ is continuous, by Proposition \ref{Fclopen}(\ref{Fclopen1}). Moreover, it is clear that $X_\star$ is an Hausdorff space, by assumption. Then, $\mathbf 1$ is an homeomorphism.
\end{proof}
\begin{cor}{\rm (\cite[Theorem 8]{folo})}\label{P-ultra=cons}
Preserve the notation of Remark \ref{esempi1}(\ref{exspec}). Then, the $\mathcal P-$ultrafilter topology and the constructible topology of $\spec(A)$ are the same. 
\end{cor}
\begin{proof}
Apply Remark \ref{esempi1}(\ref{exspec}) and Proposition \ref{coarsest}.
\end{proof}
\begin{prop}\label{functorial}
Let $X,Y$ be sets, $\mathcal F$ (resp. $\mathcal G$) be a nonempty collection of subsets of $X$ (resp. $Y$). If $f:X\longrightarrow Y$ is a function such that $\{f^{-1}(G):G\in\mathcal G\}\subseteq \mathcal F$, then $f:X^{^{\mc F-{\rm ultra}}}\longrightarrow Y^{^{\mc G-{\rm ultra}}}$ is a continuous function.
\end{prop}
\begin{proof}
Let $C$ be a closed subset of $Y$, set $\Gamma:=f^{-1}(C)$, and let $\ms U$ be an ultrafilter on $\Gamma$. Then, it sufficies to show that $\Gamma(\ms U)\subseteq \Gamma$. Let $g:f^{-1}(C)\longrightarrow C$  be the restriction of $f$ to $f^{-1}(C)$. Now, note that the collection of sets 
$\ms V:=\{D\subseteq C:f^{-1}(D)\in \ms U\}$ is an ultrafilter on $C$, since $\ms V=\ms U^g$ (see Remark \ref{beginning}(v)). Now, take an element $x\in \Gamma(\ms U)$, and fix a set $G\in\mathcal G$. If $f(x)\in G$, then  $x\in f^{-1}(G)\in \mathcal F$ (by assumption), and thus 
$
f^{-1}(G\cap C)=f^{-1}(G)\cap \Gamma\in \ms U,
$
since $x\in\Gamma(\ms U)$. This proves that, if $f(x)\in G$, then $G\cap C\in \ms V$. Conversely, if $G\cap C\in \ms V$, then $f^{-1}(G)\cap \Gamma\in\ms U$ and, since $f^{-1}(G)\in \ms U$ and $x\in \Gamma(\ms U)$, it follows $f(x)\in G$. This argument shows that $f(x)\in C(\ms V)$ and, since $C$ is closed, we have $f(x)\in C$. Then, the inclusion $\Gamma(\ms U)\subseteq \Gamma$ follows, and  the statement is now clear. 
\end{proof}
\section{Applications}
An interesting case is when $\mc F$ is a basis of a topology on $X$. 
\begin{prop}\label{B-ultra}
Let $X$ be a topological space, $\mc T$ the topology and $\mc B$  a basis of open sets of $X$. Then, the following statements hold.
\begin{enumerate}[\rm (1)]
\item The $\mc B-$ultrafilter topology on $X$ is finer than or equal to the given topology $\mathcal T$.
\item If $X$ satisfies the $T_0$ axiom, then the $\mc B-$ultrafilter topology is Hausdorff and totally disconnected. In particular, if $X$ satisfies the $T_0$ axiom but it is not Hausdorff, then the  $\mathcal B-$ultrafilter topology is strictly finer than the given topology $\mc T$.
\item Assume that $X$ satisfies the $T_0$-axiom and $X^{^{\mc B-{\rm ultra}}}$ is compact. 
\begin{enumerate}[\rm (a)]
\item Then the $\mc B-$ultrafilter topology is the coarsest topology on  $X$  for which $\mc B$ is a collection of clopen sets.
\item  Moreover $X$, equipped with the topology $\mc T$, is a spectral space.
\end{enumerate}
\end{enumerate}
\end{prop}
\begin{proof}
(1) and (2) are immediate consequences of Proposition \ref{Fclopen}. Statement (3,a) follows by applying Proposition \ref{coarsest}. 
By using  statement (3,a), it follows that $\mc S:=\mc B\cup \{X\w B:B\in\mc B\}$ is a subbasis of open sets for $X^{^{\mc B-\rm ultra}}$. Moreover, by statement (1), each member of $\mc B$ is compact, with respect to the topology $\mc T$, since it is closed (=compact) in $X^{^{\mc B-\rm ultra}}$. Now, let $X^{^{\rm patch}}$ denote the set $X$ endowed with the so-called \emph{patch topology induced by} $\mc T$, i.e. the topology whose subbasis of open sets is the set $\mc S_0$ of all the open and compact subsets of $X$, with respect to the topology $\mc T$, and their complements in $X$. The following result will be crucial for the last part of the proof. 
\begin{prop}\label{patch-referee}
Preserve the notation and the assumptions of Proposition \ref{B-ultra}(3). Then the patch topology induced by $\mc T$ is equal to the $\mc B-$ultrafilter topology. 
\end{prop}
\begin{proof}
It follows immediately that the patch topology induced by $\mc T$ is finer than or equal to the $\mc B-$ultrafilter topology (in fact $\mc S\subseteq \mc S_0$). Now, let $S_0\in\mc S_0$. Since $\mc B$ is a basis of open and compact subspaces of $X$ (with respect to $\mc T$), then there exists a finite subcollection $\mc C\subseteq \mc B$ such that $S_0=\bigcup\mc C$, or $S_0=\bigcap\mc C^-$. Thus $S_0$ is an open set of the $\mc B-$ultrafilter topology. This proves that the patch topology induced by $\mc T$ and the $\mc B-$ultrafilter topology are identical.
\end{proof}
Then, the fact that $X$, endowed with the topology $\mc T$, is a spectral space follows by applying \cite[Corollary to Proposition 7]{ho}. 
\end{proof}
\begin{cor}\label{spectralcriterion}
Let $X$ be a topological space. Then, the following conditions are equivalent.
\begin{enumerate}[\rm (i)]
\item $X$ is a spectral space.
\item There is a basis $\mc B$ of $X$ such that $X^{^{\mc B-\rm ultra}}$ is a compact and Hausdorff topological space.
\item $X$ satisfies the T$_0$-axiom and there is a basis $\mc B$ of $X$ such that $X_{\mc B}(\ms U)\neq \emptyset$, for any ultrafilter $\ms U$ on $X$.
\item $X$ satisfies the T$_0$-axiom and there is a subbasis $\mathcal S$ of $X$ such that $X_{\mc S}(\ms U)\neq \emptyset$, for any ultrafilter $\ms U$ on $X$.
\end{enumerate}
\end{cor}
\begin{proof}
(i)$\Longrightarrow$(iii) We can assume, without loss of generality, that $X=\spec(A)$, for some ring $A$. Let $\ms U$ be an ultrafilter on $X$ and $\mc P$ be the basis of $X$ made up of the principal open subsets. Keeping in mind Remark \ref{gioco}(\ref{giocospec}), we have $X_{\mc P}(\ms U)=\{\f p_{\ms U}\}$. Thus, it sufficies to choose $\mc B:=\mc P$.

(iii)$\Longrightarrow$(ii). Apply Theorem \ref{compactness} and Proposition \ref{B-ultra}(2) to the basis $\mc B$ given in condition (iii).

(ii)$\Longrightarrow$(i). It is the statement of Proposition \ref{B-ultra}(3).

(iii)$\Longleftrightarrow$(iv). It is trivial, by Proposition \ref{prebool} and Theorem \ref{compactness}.
\end{proof}

The following example will show that for fixed a topological space $X$ the $\mathcal B-$ultrafilter topology depends on the choice of the basis $\mc  B$. 
\begin{ex}
Let $K$ be a field, $\{T_n:n\in\mathbf N\}$ be an infinite and countable collection of indeterminates over $K$ and consider the ring $A:=K[\{T_n:n\in\mathbf N\}]$. Set $X:=\spec(A)$ and endow this set with the Zariski topology. As usual, let $\mc P:=\{D_f:f\in A\}$ be the basis of the principal open subsets of $\spec(A)$, and let $\mc T:=\{D(\f a):=X\w V(\f a): \f a \mbox{ ideal of }A\}$ (clearly, $\mc T$ is a basis of $X$, being it the topology). We claim that the $\mc P-$ultrafilter topology (i.e. the usual ultrafilter topology on  $X$) and the $\mc T-$ultrafilter topology are different. Let $\f m$ be the maximal ideal of $A$ generated by the set $\{T_n:n\in\mathbf N\}$ and set $\mc F:=\{V(T_n):n\in \mathbf N\}\cup\{X\w \{\f m\}\}$. It is straightforward that $\mc F$ is a collection of subsets of $X$ with the finite intersection property, and thus there exists an ultrafilter $\ms U$ on $X$ containing $\mc F$, by virtue of Remark \ref{beginning}(iii). We claim that the set 
$$
X_{\mc T}(\ms U):=\{\f p\in X:\mbox{for each ideal }\f a\mbox{ of }A ,( \f p\in D(\f a)\Longleftrightarrow D(\f a)\in \ms U)\}
$$
is empty. If not, let $\f p\in X_{\mc T}(\ms U)$. Since $\mc F\subseteq \ms U$, it follows that $V(T_n)\in \ms U$, for each $n\in\mathbf N$, and thus $T_n\in \f p$, for each $n\in\mathbf N$ (by the definition of $X_{\mc T}(\ms U)$). This proves that $\f p=\f m$. On the other hand, if we set $\f a:=\f m$, we have obviously $\f m\notin D(\f a)$, hence $D(\f a)\notin \ms U$, that is, $V(\f a)=\{\f m\}\in \ms U$. It follows $\emptyset\in\ms U$, since $X\setminus \{\f m\}\in \mc F\subseteq \ms U$, a contradiction. This argument proves that $X_{\mc T}(\ms U)$ is empty, and thus $X^{^{\mc T-{\rm ultra}}}$ is not compact, by Theorem \ref{compactness}. It follows that the $\mc T-$ultrafilter topology and the $\mc P-$ultrafilter topology on $X$ are not the same, since the $\mc P-$ultrafilter topology is compact.                                             
\end{ex}\label{basis}
\begin{prop}\label{overrings}
Let $B|A$ be a ring extension and ${\rm S}(B|A)$ be the set of all the rings $C$ such that $A\subseteq C\subseteq B$. For each $x\in B$ set 
$$
U_x:=\{C\in {\rm S}(B|A):x\in C\}
$$
Let ${\rm S}(B|A)$ be endowed with the Zariski topology, i.e. the topology for which the collection $\mc R:=\{U_x:x\in B\}$ is a subbasis of open sets.  
The following statements hold.
\begin{enumerate}[\rm (1)]
\item If $Y$ is a subset of ${\rm S}(B|A)$ and $\ms U$ is an ultrafilter on $Y$, then $A_{\ms U}:=\{x\in B:U_x\cap Y\in \ms U\}$ belongs to ${\rm S}(B|A)$. 
\item ${\rm S}(B|A)$ is a spectral space, endowed with the Zariski topology. 
\end{enumerate}
\end{prop}
\begin{proof}
(1). Let $\ms U$ be an ultrafilter on $Y$. The fact that $A_{\ms U}$ is a ring is proved by using the same argument given in  \cite{calota}. Moreover, for each $x\in A$, $U_x\cap Y=Y$ and thus it belongs to $\ms U$. Now, the inclusion $A\subseteq A_{\ms U}$ is clear.

(2). Let $\ms U$ be an ultrafilter on $X:={\rm S}(B|A)$. Keeping in mind (1) and  definitions, it is straightforward to verify that $X_{\mc R}(\ms U)=\{A_{\ms U}\}$. Then the conclusion follows immediately by Corollary \ref{spectralcriterion}.
\end{proof}
\begin{prop}\label{icoverrings}
Let $B|A$ be a ring extension and ${\rm I}(B|A)$ be the set of all the rings $C\in {\rm S}(B|A)$ such that $C$ is integrally closed in $B$. As usual, define on ${\rm I}(B|A)$ the Zariski topology by taking as subbasis of open sets the collection 
$$
\mc R':=\{U_x\cap {\rm I}(B|A):x\in C\}
$$
(i.e. this is the subspace topology, induced by the Zariski topology on ${\rm S}(B|A)$). Then the following statements hold.
\begin{enumerate}[\rm (1)]
\item If $Y$ is a subset of ${\rm I}(B|A)$ and $\ms U$ is an ultrafilter on $Y$, then $A_{\ms U}:=\{x\in C:Y\cap U_x\in \ms U\}$ belongs to ${\rm I}(B|A)$. 
\item ${\rm I}(B|A)$ is a spectral space, endowed with the Zariski topology. 
\end{enumerate}
\end{prop}
\begin{proof}
(1). Proposition \ref{overrings}(1) implies  $A_{\ms U}\in {\rm S}(B|A)$. Now, let $x\in B$ be an element integral over $A_{\ms U}$. Pick elements  $a_0,\z,a_{n-1}\in A_{\ms U}$  such that 
$$
x^n+a_{n-1}x^{n-1}+\z+a_1x+a_0=0.
$$
If $C\in Y\cap \bigcap_{i=0}^{n-1}U_{a_i}$, then $x$ is integral over $C$. Since $Y\subseteq {\rm I}(B|A)$, it follows $x\in C$. This argument shows that $Y\cap \bigcap_{i=0}^{n-1}U_{a_i}\subseteq Y\cap U_x$. Keeping in mind that $a_0,\z,a_{n-1}\in A_{\ms U}$, it follows $Y\cap \bigcap_{i=0}^{n-1}U_{a_i}\in\ms U$, and thus $Y\cap U_x\in \ms U$. Then $x\in A_{\ms U}$, and this proves that $A_{\ms U}$ is integrally  closed in $B$. 

(2) Set $X:={\rm I}(B|A)$. As in Proposition \ref{overrings}(2), it is immediately verified that $X_{\mc R'}(\ms U)=\{A_{\ms U}\}$, for any ultrafilter $\ms U$ on $X$. Then, it suffices to apply Corollary 3.2.
\end{proof}

\begin{prop}\label{comp-gen-clo}
Let $X$ be a topological space  and $\mathcal F$ be a collection of subsets of $X$ containing at least a basis of open sets of $X$. If $Y$ is a compact subspace of $X$, then the generic closure of $Y$ (with respect to the given topology) is closed, with respect to the $\mc F-$ultrafilter topology.  
\end{prop}
\begin{proof}
Let us denote by $Y_0$ the generic closure of $Y$ and assume, by contradiction, that $Y_0$ is not a closed subset of $X^{\mc F-{\rm ultra}}$. By definition, there exist an ultrafilter $\ms U$ on $Y_0$ and a point $x_0\in Y_0(\ms U)\w Y_0$. Let $\mc B$ be a basis of $X$ contained in $\mc F$. 
By the definition of the set $Y_0$, for each  $y\in Y$ there exists an open neighborhood $\Omega_y$ of $y$ such that $x_0\notin\Omega_y$. Without loss of generality, we can assume that $\Omega_y\in \mc B$. By compactness, the open cover $\{\Omega_y:y\in Y\}$ of $Y$ admits a finite subcover, say $\{\Omega_{y_i}:i=1,\z,n\}$. Now, by definition, it is immediately verified that $Y_0\subseteq \bigcup\{\Omega_{y_i}:i=1,\z,n\}$, that is $Y_0=\bigcup\{\Omega_{y_i}\cap Y_0:i=1,\z,n\}$, and thus, by Remark \ref{beginning}(2), it follows that $\Omega_{y_i}\cap Y_0\in \ms U$, for some $i\in\{1,\z,n\}$. Keeping in mind that, by assumption, $\Omega_{y_i}\in \mathcal F$ and that $x_0\in Y_0(\ms U)$, it follows $x_0\in \Omega_{y_i}$, a contradiction. 
\end{proof}
Finally, we will show that the so called constructible topology of a quasi-separated scheme (see \cite[pag. 337, (7.2.11)]{EGA}) can be presented as a particular case of the construction given in Section 2. 
\begin{thm}\label{scheme-constructible}
Let $X$ be the underlying topological space of a quasi-separated scheme, and let $\mathcal K$ be the collection of all the open and compact subspaces of $X$. Then, $\mathcal K$ is a basis of $X$ and the constructible topology on $X$ is equal to the $\mathcal K-$ultrafilter topology. 
\end{thm}
\begin{proof}
Let $\mathcal A$ be any open cover of $X$ that consists only of spectral subspaces of $X$ (such an open cover exists, since $X$ is the underlying topological space of a scheme). For any $U\in\mathcal A$, let $\mathcal B_U$ be the basis of $U$ made up of the open and compact subspaces of $U$, and set $\mathcal B:=\bigcup\{\mathcal B_U:U\in\mathcal A\}$. Then $\mathcal B$ is a basis of $X$, and it is straightforward that $\mathcal B\subseteq \mathcal K\subseteq {\rm Bool}(\mathcal B)$ (in particular, also $\mathcal K$ is a basis of $X$). Thus, in view of Corollary \ref{bool}, the $\mathcal K-$ultrafilter topology and the $\mathcal B-$ultrafilter topology on $X$ are the same. 
 Now, let $Y$ be a subset of $X$. Keeping in mind \cite[Proposition (7.2.3)(iv)]{EGA}, it sufficies to show that the following conditions are equivalent:
\begin{enumerate}[(i)]
\item $Y$ is $\mathcal B-$stable under ultrafilters.
\item For any $U\in\mathcal A$, $Y\cap U$ is closed, with respect to the constructible topology of $U$.
\item[(ii)'] For any $U\in\mathcal A$, $Y\cap U$ is a $\mathcal B_U-$stable under ultrafilters subset of $U$.
\end{enumerate}
Since the equivalence of (ii) and (ii)' follows immediately by Corollary \ref{P-ultra=cons},  it is enough to show that (i) is equivalent to (ii)'. 

(ii)'$\Longrightarrow$(i). Let $\ms U$ be an ultrafilter on $Y$ and let $x_0\in Y_{(X,\mathcal B)}(\ms U)$. We want to show that $x_0\in Y$. Pick a spectral subspace $U\in \mathcal A$ containing $x_0$, set $T:=Y\cap U$, and consider the ultrafilter $\ms U_T:=\{S\cap T:S\in\ms U\}$ on $T$ (see Remark \ref{beginning}(\ref{giu})). Keeping in mind that $T$ is a $\mathcal B_U-$stable under ultrafilter subset of $U$ and applying Lemma \ref{technical} and Remark \ref{esempi1}(\ref{finer}), we have
$$
x_0\in Y_{(X,\mathcal B)}(\ms U)\cap U\subseteq T_{(X,\mathcal B_U)}(\ms U_T)\cap U=:T_{(U,\mathcal B_U)}(\ms U_T)\subseteq T
$$
and thus $x_0\in Y$. 

(i) $\Longrightarrow$ (ii)'. Suppose that $Y$ is $\mathcal B-$stable under ultrafilters, fix a spectral subspace $U\in \mathcal A$, and let $\ms U$ be an ultrafilter on $Z:=U\cap Y$. We want to show that $Z_{(U,\mathcal B_U)}(\ms U)\subseteq Z$. Take an element $x_0\in Z_{(U,\mathcal B_U)}(\ms U)$ and consider the ultrafilter $\ms U^Y:=\{S\subseteq Y:S\cap Z\in \ms U\}$ on $Y$ (Remark \ref{beginning}(\ref{su})). Since, by assumption, we have 
$$
Y_{(X,\mathcal B)}(\ms U^Y):=\{x\in X:[\forall B\in \mathcal B, x\in B\Longleftrightarrow B\cap Y\in \ms U^Y]\}\subseteq Y
$$
it sufficies to show that $x_0\in Y_{(X,\mathcal B)}(\ms U^Y)$. Fix an open set $B\in \mathcal B$. If $x_0\in B$, we can pick an open and compact subspace $\Omega\in \mathcal B_U$ of $U$ such that $x_0\in \Omega\subseteq B\cap U$. Since $x_0\in Z_{(U,\mathcal B_U)}(\ms U)$, it follows $(\Omega\cap Y)\cap Z=\Omega \cap Z\in \ms U$, by definition, and thus $\Omega \cap Y\in \ms U^Y$. Moreover $B\cap Y\in\ms U^Y$, since $\Omega\cap Y\subseteq B\cap Y$. Conversely, suppose that $B\cap Y\in \ms U^Y$, i.e. $B\cap Y\cap U\in \ms U$. Since $X$ is the underlying topological space of a quasi-separated scheme, $B\cap U$ is an open and compact subspace of $U$ \cite[Proposition (6.1.12)]{EGA}, that is $B\cap U\in \mathcal B_U$. Thus we have $x_0\in B\cap U$, since $x_0\in Z_{(U,\mathcal B_U)}(\ms U)$. The proof is now complete. 
\end{proof}
\begin{center}
\textsc{Ackowledgements}
\end{center}
The author is grateful to the referee for his/her useful comments which greatly improved the paper. 

\end{document}